\documentclass[a4paper]{amsart}

\usepackage{amsmath,amssymb,amsthm}
\usepackage[mathcal]{euscript}

\usepackage[latin1]{inputenc}

\input cyracc.def \font\tencyr=wncyr10 \def\russe{\tencyr\cyracc} 
\def\Sha{\text{\russe{Sh}}}

\newtheorem{theorem}{Theorem}

\newtheorem{proposition}[theorem]{Proposition}
\newtheorem{lemma}[theorem]{Lemma}
\newtheorem{corollary}[theorem]{Corollary}
\newtheorem{conjecture}{Conjecture}

\renewcommand{\geq}{\geqslant}
\renewcommand{\leq}{\leqslant}

\newcommand{\Q}{\mathbb Q}
\newcommand{\N}{\mathbb N}
\newcommand{\F}{\mathbb F}
\newcommand{\R}{\mathbb R}
\newcommand{\Z}{\mathbb Z}
\newcommand{\C}{\mathbb C}

\newcommand\qbi[3]{{{#1}\atopwithdelims[]{#2}}_{#3}}

\newcommand{\la}{\lambda}

\newcommand{\ds}{\displaystyle}

\DeclareMathOperator{\cl}{C\ell}
\DeclareMathOperator{\rk}{rk}
\DeclareMathOperator{\disc}{disc}

\DeclareMathOperator{\sel}{Sel} 

\begin{document}  
    
\title{The Cohen-Lenstra heuristics, moments and $p^j$-ranks of some groups}

\author{Christophe Delaunay}
\address{Universit\'e de Franche-Comt\'e \\
Laboratoire de Math\'ematiques de Besan\c con, CNRS UMR UMR 6623 \\
Facult\'es des Sciences et Techniques \\ 
16 route de Gray \\
25030 Besan\c con  \\
France.}
\email{christophe.delaunay@univ-fcomte.fr}
\author{Fr\'ed\'eric Jouhet}
\address{Universit\'e de Lyon \\
CNRS  \\
Universit\'e Lyon 1 \\
Institut Camille Jordan \\
43, boulevard du 11 novembre 1918 \\
F-69622 Villeurbanne Cedex \\
France.}
\email{jouhet@math.univ-lyon1.fr}
\date{\today}

\begin{abstract}
This article deals with the coherence of the model given by the Cohen-Lenstra heuristic philosophy for class groups and also for their generalizations to 
Tate-Shafarevich groups. More precisely, our first goal is  to extend a previous result due to E. Fouvry and J. Kl\"uners which proves that a conjecture 
provided by the Cohen-Lenstra philosophy implies another such conjecture. As a consequence of our work, we can deduce, for example, a conjecture 
for the probability laws of $p^j$-ranks of Selmer groups of elliptic curves. This is compatible with some theoretical works and other classical 
conjectures.
\end{abstract}

\maketitle

%
%

\section{Introduction and notation}
\noindent
The Cohen-Lenstra heuristics and their generalizations to Tate-Shafarevich groups are models for formulating conjectures related to class groups of number fields and 
Tate-Shafarevich groups of elliptic curves varying in some natural families. This article deals with the coherence of 
the model. More precisely,  our aim is to prove that a conjecture provided by the Cohen-Lenstra philosophy implies another such conjecture. 
This work actually extends and generalizes an earlier one by E.~Fouvry and J.~Kl\"uners in~\cite{kluners_fouvry2} which deals with class groups; we will 
follow their presentation and adapt the main techniques of their proofs. \medskip
~\\ 
We will use the following notation.  The letter $d$ will denote a fundamental discriminant and $\cl(K_d)$ the class group associated to the quadratic 
number field $K_d=\Q(\sqrt{d})$. The letter $p$ will always denote a prime number. If $G$ is a finite abelian group, the $p^j$-rank of $G$ is defined by 
$\rk_{p^j}(G)=\dim_{\F_p} p^{j-1}G/p^jG$. For any real valued function $f$ defined over isomorphism classes of finite abelian groups, we say that 
$f(\cl(K_d))$ has average value $c_\pm \in \R$ if 
$$
\frac{\ds \sum_{0<\pm d<X} f(\cl(K_d))}{\ds \sum_{0<\pm d<X} 1}= c_\pm +o(1), \quad  \mbox{ as } X \rightarrow \infty.
$$
If $f$ is the characteristic function of some property, then $c_\pm$ is called the probability of the property (or the density of the set of the class groups 
satisfying it). 
~\\
The Cohen-Lenstra heuristics allow one to formulate conjectures for the average values of $f(\cl(K_d)_{odd})$ for ``reasonable'' functions $f$, where 
$\cl(K_d)_{odd}$ 
denotes the odd part of $\cl(K_d)$. Concerning 
the $2$-part of class groups, genus theory enables one to determine the structure 
of $\cl(K_d)[2]$. However Gerth (\cite{gerth_1}, \cite{gerth_2}) succeeded in generalizing the Cohen-Lenstra heuristics for the $4$-part of class groups, 
and also obtained conjectures for the average values of $f(\cl(K_d)^2)$ for reasonable functions $f$.\medskip
~\\
In particular, the above heuristics give on the one hand a prediction for the average value of the function 
$f(p^{k\rk_p(\cl(K_d)^2)})$ for all $k\in \N$, and on the other hand a prediction for the probability that 
$\rk_p(\cl(K_d)^2)=r$  for all $r\in \N$ (note that for odd $p$, $\rk_p(\cl(K_d)^2) = \rk_p(\cl(K_d))$). 
In \cite{kluners_fouvry2}, the authors proved that the former prediction implies the latter. The aim of this article is to generalize the results 
of \cite{kluners_fouvry2} in two directions. First, we will consider higher moments (including for example the average values of 
$f(p^{k\rk_{p^j}(\cl(K_d)^2)})$ 
for all positive integers $(k,j)$) and probability laws of $p^j$-ranks for $j \geq1 $. Secondly, we will obtain analogous results concerning heuristics on Tate-Shafarevich groups and on 
Selmer groups  of elliptic curves~(\cite{delaunay1, delaunay2, delaunay_jouhet}).\bigskip
~\\
We recall for $(a,q)\in \C^2$ with $|q|<1$ and  $k\in \Z$ the $q$-shifted factorial
\begin{equation*}
(a;q)_k:=\left\{\begin{array}{l}1\;\;\mbox{if}\;\;k=0\\ (1-a)\dots
(1-aq^{k-1})\;\;\mbox{if}\;\;k>0\\1/(1-aq^{-1})\dots
(1-aq^{k})\;\;\mbox{if}\;\;k<0,
\end{array}\right.
\end{equation*}
and $(a;q)_\infty:=\lim_{k\to+\infty}(a;q)_k$. Note that $(1/p;1/p)_k=\eta_k(p)$, where $\eta_k$ is the function defined in \cite{cohen-lenstra} and used in \cite{kluners_fouvry2}. We will also use 
the $q$-binomial coefficient 
$$
\qbi{n}{k}{q}:=\frac{(q;q)_n}{(q;q)_k(q;q)_{n-k}}\in\mathbb{N}[q].
$$
\medskip
~\\
A partition $\la:=(\la_1\geq\la_2\geq\cdots)$ of a nonnegative integer $n$ is a finite decreasing sequence of nonnegative integers 
whose sum is equal to $n$. If $\lambda$ is a partition of $n$, we write $|\lambda|=n$ and the notation 
$\lambda=1^{m_1}2^{m_2}\cdots \ell^{m_\ell}$ means that $m_i$ is the multiplicity of the integer $i$ in $\lambda$ (hence, we have 
$n=\lambda_1 + \lambda_2+\cdots = |\lambda|=m_1+2m_2+\cdots + \ell m_\ell$). If $\mu:=(\mu_1 \geq \mu_2\geq \cdots)$ is a second integer partition, 
then we define $(\lambda|\mu):=\sum_i \lambda_i \mu_i$ (we will often use the statistics $(\lambda | \lambda)=\sum_i \lambda_i^2$ which must not be 
mistaken for $|\lambda|^2=(\sum_i \lambda_i)^2$). Finally, the notation $\mu \subseteq \lambda$ means that $\mu_i \leq \lambda_i$ for all $i\geq1$.
\\
Recall that a finite abelian $p$-group $G$ has type $\lambda=1^{m_1}\cdots \ell^{m_\ell}$ if 
$$
G\simeq \left(\Z/p\Z\right)^{m_1} \oplus \cdots \oplus \left(\Z/p^\ell\Z\right)^{m_\ell}.
$$
If $\lambda=1^{m_1}\cdots \ell^{m_\ell}$ is an integer partition, we denote by $\lambda':=(\lambda'_1\geq \lambda'_2 \geq \cdots)$ its conjugate defined 
by $\lambda'_k=\sum_{j=k}^\ell m_j$ for all $k$. We have $|\lambda|=|\lambda'|$.  \medskip
~\\
As in \cite{delaunay_jouhet}, we denote by $C_{\lambda,\mu}(p)$ the number of subgroups of type $\mu$ in a finite
abelian $p$-group of type $\lambda$, which can be expressed by
\begin{equation}\label{coeff}
C_{\la,\,\mu}(p)=p^{\sum_{i\geq1}\mu'_{i+1}(\la'_i-\mu'_i)}
\,\prod_{i\geq1}\qbi{\la'_i-\mu'_{i+1}}{\la'_i-\mu'_i}{p},
\end{equation}
showing that it is a polynomial in the variable $p$, with positive integral coefficients.
\medskip
~\\
Using the Cohen-Lenstra philosophy (\cite{cohen-lenstra}) and a combinatorial analysis, we obtained the following conjecture (\cite[Conjecture 1]{delaunay_jouhet}).
\begin{conjecture}\label{moment_class_gp}
For any positive integer $\ell$, let $\lambda=1^{m_1}2^{m_2}\cdots \ell^{m_\ell}$ be an integer partition, and assume\footnote{In order to simplify 
the notations, we will exclude $p=2$ for class groups. However, one can obtain directly the case $p=2$ by  replacing  $\cl(K_d)$ by $\cl(K_d)^2$ all along the article.} 
that $p \geq 3$. As $d$ is varying over the set of fundamental negative discriminants, the average of 
$|\cl(K_d)[p]|^{m_1} |\cl(K_d)[p^2]|^{m_2} \cdots |\cl(K_d)[p^\ell]|^{m_\ell}$ is equal to
$$
\sum_{\mu \subseteq \la} C_{\la,\mu}(p),
$$
where the sum is over all integer partitions $\mu \subseteq \lambda$.~\\
Similarly, as $d$ is varying over the set of fundamental positive discriminants, the average of 
$|\cl(K_d)[p]|^{m_1} |\cl(K_d)[p^2]|^{m_2} \cdots |\cl(K_d)[p^\ell]|^{m_\ell}$ 
is equal to
$$
\sum_{\mu \subseteq \la} C_{\la,\mu}(p) p^{-|\mu|}.
$$
\end{conjecture}
\medskip
~\\
Concerning the probability laws of $p^j$-ranks $\rk_{p^j}(\cl(d))$, the following conjecture comes naturally from \cite[Corollary 11]{delaunay3}.
\begin{conjecture}\label{rk_laws_gp}
Let $\ell$ be a positive integer and $\mu:=\mu_1\geq \mu_2\geq \cdots \geq \mu_\ell \geq 0$ a partition of length $\ell(\mu) \leq \ell$ (i.e., $\mu_{\ell+1}=0$). Assume that
$p\geq 3$. Then, as $d$ is varying over 
the set of fundamental negative discriminants, the probability that $\rk_{p^j}(\cl(K_d))=\mu_j$ for all $1\leq j \leq \ell$ is equal to
$$
\frac{(1/p^{\mu_\ell+1};1/p)_\infty}{p^{\mu_1^2+\cdots + \mu_\ell^2} \prod_{j=1}^{\ell} (1/p;1/p)_{\mu_j-\mu_{j+1}}}.
$$
Moreover, as $d$ is varying over the set of fundamental positive discriminants, the probability that $\rk_{p^j}(\cl(K_d))=\mu_j$ for all $1\leq j \leq \ell$ is equal to
$$
\frac{(1/p^{\mu_\ell+2};1/p)_\infty}{p^{\mu_1^2+\cdots + \mu_\ell^2+(\mu_1+\cdots + \mu_\ell)} \prod_{j=1}^{\ell} (1/p;1/p)_{\mu_j-\mu_{j+1}}}.
$$
\end{conjecture}
\medskip
~\\
Very few is known about these conjectures. Davenport and Heilbronn (\cite{davenport_heilbronn}) proved Conjecture \ref{moment_class_gp} for 
$p=3$ and $\lambda=1^1$. In \cite{kluners_fouvry}, the authors proved Conjecture~\ref{moment_class_gp} for $p=2$ (replacing $\cl(K_d)$ 
by $\cl(K_d)^2$) and any $\lambda=1^{m_1}$.
~\\
The conjectures mentioned in the introduction coming  from the seminal work of \cite{cohen-lenstra} and studied 
in \cite{kluners_fouvry, kluners_fouvry2} correspond to Conjecture~\ref{moment_class_gp} for $\lambda = 1^{m_1}$ and 
Conjecture~\ref{rk_laws_gp} for $\ell=1$. More precisely, if $\lambda=1^{m}$, Conjecture~\ref{moment_class_gp} says that the average 
of $|\cl(K_d)[p]|^m$ for class groups of imaginary 
(resp. real) quadratic fields is equal to (with the notations of \cite{delaunay_jouhet})
$$
M_0(x^{1^n}) = \sum_{k=0}^n \qbi{n}{k}{p} \quad \left(\mbox{resp. } M_1(x^{1^n}) = \frac{1}{p}\sum_{k=0}^n \qbi{n}{k}{p} \right) 
$$
Those correspond to $\mathcal{N}(n,p)$ (resp. $\mathcal{M}_n(p)$) used in \cite{kluners_fouvry, kluners_fouvry2}. Fouvry and Kl\"uners proved in 
\cite{kluners_fouvry2} that if Conjecture~\ref{moment_class_gp} is true with  $\lambda=1^m$ for all $m$, then Conjecture~\ref{rk_laws_gp} is true 
with $\ell=1$ for all $\mu_1\geq 0$. We will adapt their proof and use a result in \cite{delaunay_jouhet} to simplify one of its step in order to obtain the following generalization.
\begin{theorem}\label{class-group} Assume that $p\geq 3$. Let $\ell$ be a positive integer and assume that for any $\lambda=1^{m_1}2^{m_2}\cdots \ell^{m_\ell}$, as $d$ is varying over 
the set of fundamental negative discriminants,
the average of $|\cl(K_d)[p]|^{m_1} |\cl(K_d)[p^2]|^{m_2} \cdots |\cl(K_d)[p^\ell]|^{m_\ell}$ is equal to $\sum_{\mu \subseteq \la} C_{\la,\mu}(p)$. Then for any 
$\mu_1\geq \mu_2 \geq \cdots \geq \mu_\ell\geq0$, as  $d$ is varying over the set of fundamental negative discriminants, the probability that 
$\rk_{p^j}(\cl(K_d))=\mu_j$ for all $1\leq j \leq \ell$ is equal to
$$
\frac{(1/p^{\mu_\ell+1};1/p)_\infty}{p^{\mu_1^2+\cdots + \mu_\ell^2} \prod_{j=1}^{\ell} (1/p;1/p)_{\mu_j-\mu_{j+1}}}.
$$
Furthermore, assume that for any positive integer $\ell$ and any $\lambda=1^{m_1}2^{m_2}\cdots \ell^{m_\ell}$, as $d$ is varying over 
the set of fundamental positive discriminants, the average of 
$|\cl(K_d)[p]|^{m_1} |\cl(K_d)[p^2]|^{m_2} \cdots |\cl(K_d)[p^\ell]|^{m_\ell}$ is equal to $\sum_{\mu \subseteq \la} C_{\la,\mu}(p) p^{-|\mu|}$. Then for any 
$\mu_1\geq \mu_2 \geq \cdots \geq \mu_\ell$, as $d$ is varying over 
the set of fundamental positive discriminants, the probability that 
$\rk_{p^j}(\cl(K_d))=\mu_j$ for all $1\leq j \leq \ell$ is equal to
$$
\frac{(1/p^{\mu_\ell+2};1/p)_\infty}{p^{\mu_1^2+\cdots + \mu_\ell^2+(\mu_1+\cdots + \mu_\ell)} \prod_{j=1}^{\ell} (1/p;1/p)_{\mu_j-\mu_{j+1}}}.
$$
\end{theorem}
\medskip
~\\
One can also adapt the Cohen-Lenstra heuristics for Tate-Shafarevich groups of elliptic curves. If $E$ is an elliptic curve defined over $\Q$, we denote by $\Sha(E)$ its Tate-Shafarevich group.
In this context, {\bf we assume 
in this article that $\Sha(E)[p^\infty]$ is finite for all elliptic curves $E/\Q$} (which is a classical conjecture). In that case, $\Sha(E)[p^\infty]$ is a group of 
type\footnote{This notion is different from the previously mentioned groups of type $\lambda$, where $\lambda$ is a partition.} S 
(i.e. it is endowed with a bilinear, alternating, non degenerate pairing $\beta \colon \Sha(E)[p^\infty] \times \Sha(E)[p^\infty] \rightarrow \Q/\Z$).
Let ${\mathcal F}_u$ be the family of 
elliptic curves $E$ defined over\footnote{One can replace in our discussion $\Q$ by any other number field $K$.} $\Q$ with rank $u$, ordered by 
their conductor (denoted by $N(E)$).  If $f$ is a real valued function defined over isomorphism classes of groups of type
 S (see \cite{delaunay1,delaunay2}), 
then we say that $f(\Sha(E))$ has average value $c \in \R$ for $E$ varying over ${\mathcal F}_u$ if  
$$
\sum_{\substack{E\in { \mathcal F}_u \\ N(E)<X}} f(\Sha(E))= (c+o(1)) \sum_{\substack{E\in { \mathcal F}_u \\ N(E)<X}} 1 , \quad  \mbox{ as } X \rightarrow \infty.
$$
If $f$ is the characteristic function of some property, we say that $c$ is the probability of this property (or the density of the set of Tate-Shafarevich  groups satisfying it) for $E$ 
varying over ${\mathcal F}_u$. We raised the following conjecture in \cite{delaunay_jouhet}.
\begin{conjecture}\label{moment_sha}
Let $\ell$ be a positive integer, let $\lambda=1^{m_1}2^{m_2}\cdots \ell^{m_\ell}$ be a partition, and let $u$ be a nonnegative integer. As $E/\Q$, 
ordered by conductors, is varying over ${\mathcal F}_u$, the average of 
$|\Sha(E)[p]|^{m_1} |\Sha(E)[p^2]|^{m_2} \cdots |\Sha(E)[p^\ell]|^{m_\ell}$ is equal to
$$
\sum_{\mu \subseteq \lambda} C_{\lambda,\mu}(p^2)p^{-|\mu|(2u-1)}.
$$
\end{conjecture}
\medskip
~\\
Concerning the probability laws of $p^j$-ranks $\rk_{p^j}(\Sha(E))$, we have the following (\cite{delaunay3}).
\begin{conjecture}\label{rk_laws_sha}
Let $\ell$ be a positive integer, let $\mu=\mu_1\geq \mu_2\geq \cdots \geq \mu_\ell \geq 0$ be an integer partition of length $\ell(\mu) \leq \ell$  
and let $u$ be a nonnegative integer. As $E/\Q$, ordered by conductors, is varying over ${\mathcal F}_u$, the probability that  $\rk_{p^j}(\Sha(E))=2\mu_j$ 
for all $1\leq j \leq \ell$ is equal to
$$
\frac{(1/p^{2u+2\mu_\ell+1};1/p^2)_\infty}{p^{2(\mu_1^2+\cdots + \mu_\ell^2)+(2u-1)(\mu_1+\cdots + \mu_\ell)} \prod_{j=1}^{\ell} (1/p^2;1/p^2)_{\mu_j-\mu_{j+1}}}.
$$
\end{conjecture}
\medskip
~\\
As previously, very few is known about these conjectures. Bhargava and Shankar \cite{bhargava_shankar_1, bhargava_shankar_2} obtained 
several results about the average of $|S(E)_p|$ over all elliptic curves $E/\Q$, where $S(E)_p$ is the $p$-Selmer group of $E/\Q$. Their results, 
together with a strong form of the rank conjecture 
(asserting that the rank of $E$ is $0$ or $1$ 
with probability $1/2$ each and that elliptic curves with rank $\geq 2$ do not contribute in the averages), imply Conjecture~\ref{moment_sha} for 
$\lambda=1^1$ and $p=2,3$. Heath-Brown \cite{heath-brown,heath-brown2}, then Swinnerton-Dyer \cite{swinnerton-dyer}  and Kane \cite{kane} also obtained results about $S(E)_2$ when 
$E$ is varying over some families of quadratic twists. 
Their results, together with a strong rank conjecture, imply Conjecture~\ref{moment_sha} 
for $\lambda=1^1$ and  $p=2$ for some families of quadratic twists. Furthermore, Conjecture~\ref{moment_sha} is compatible with the 
conjecture of Poonen and Rains (\cite{poonen-rains}).\medskip
~\\
In this context of elliptic curves, we will prove the following result.
\begin{theorem}\label{general_sha} 
Let $u$ be a nonnegative integer, and let $\ell$ be a positive integer. Assume that for 
any $\lambda=1^{m_1}2^{m_2}\cdots \ell^{m_\ell}$, as $E/\Q$, ordered by conductors, is varying over ${\mathcal F}_u$, 
the average of $|\Sha(E)[p]|^{m_1} |\Sha(E)[p^2]|^{m_2} \cdots |\Sha(E)[p^\ell]|^{m_\ell}$ is equal to 
$\sum_{\mu \subseteq \la} C_{\lambda,\mu}(p^2)p^{-|\mu|(2u-1)}$. Then for any $\mu_1\geq \mu_2\geq \cdots \geq \mu_\ell$, as $E/\Q$ is varying over ${\mathcal F}_u$, the probability that 
$\rk_{p^j}(\Sha(E))=2\mu_j$ for 
all $1\leq j \leq \ell$ is equal to
$$
\frac{(1/p^{2u+2\mu_\ell+1};1/p^2)_\infty}{p^{2(\mu_1^2+\cdots + \mu_\ell^2)+(2u-1)(\mu_1+\cdots + \mu_\ell)} \prod_{j=1}^{\ell} (1/p^2;1/p^2)_{\mu_j-\mu_{j+1}}}.
$$
\end{theorem}
\medskip
~\\
{\bf Remarks.}
~\\
1- In Theorem \ref{class-group} (resp. Theorem \ref{general_sha}) one can replace class groups (resp. Tate-Shafarevich groups) 
by finite abelian groups (resp. groups of type S) varying in some families. In particular, we can obtain similar results for 
Selmer groups of elliptic curves (see section~\ref{Selmer}). \smallskip
~\\
2- From the laws of the $p^j$-ranks for all $j=1,\dots,\ell$ as in theorems~\ref{class-group} and~\ref{general_sha}, one can also deduce the probability 
that the $p^\ell$-rank is equal to some fixed value for a single~$\ell$. In this case, we recover the results and conjectures from \cite{cohen3} and \cite{delaunay3}. 
\smallskip
~\\
3- One can also ask if it is possible to deduce the moments from the probability laws of the $p^j$-ranks for all $j=1,\dots, \ell$. For this, it seems that 
we need to know an error term for the probability laws and this error term is not given by the heuristic philosophy. However, the theoretical 
results of \cite{kluners_fouvry, heath-brown, swinnerton-dyer, kane} concern the $p^j$-ranks of the groups studied (with an explicit error term) 
from which the moments are deduced.   
\section{An auxiliary analytic tool}
In this section, we prove a generalization of \cite[Lemma 6]{kluners_fouvry2} which will be useful later.
\begin{lemma} Let $a\in \C$ with $|a|>1$ and $g(z)=\sum_{r\geq 0} w_r z^r$ be an entire function satisfying the following properties:
\begin{itemize}
\item there exists an absolute constant $C>0$ and $\alpha\in \R$ such that for all $r\in\mathbb{N}$, $|w_r| \leq C a^{-r^2/2+\alpha r}$;
\item for all $m\in \N$, $g(a^m)=0$.
\end{itemize}
We denote by $\omega\in \N \cup \{\infty \}$ the vanishing order of $g$ at $z=0$. Then, if $\omega > \alpha -1/2$, we have $g\equiv 0$ (i.e. 
$\omega=\infty$).
\end{lemma}
\begin{proof}
Let $k$ be a nonnegative integer. Taking $|z|=|a|^k$, a direct computation shows that  $|g(z)| \leq C' |a|^{(k+\alpha)^2/2}$, where 
$C'=C \sum_{r\in \Z} |a|^{-(r-\alpha)^2/2}$. Assume that $g \not \equiv 0$.
Then \cite[Lemma 6]{kluners_fouvry2} gives
$$
\sup_{|z|=|a|^k} |g(z)| \gg |a|^{k(k+1)/2 + k\omega}.
$$
Hence, we must have $(k+\alpha)^2 \geq k(k+1) +2k\omega$ for all $k\in \N$, which implies $\omega \leq \alpha -1/2$.
\end{proof}
%
%
%
\begin{corollary}\label{coro_2}
Let $\ell \in \N^*$, let $a\in \C$ with $|a|>1$ and let $g(\underline{z}) = \sum_{r} w_r z_1^{r_1}z_2^{r_2}\cdots z_\ell^{r_\ell}$ with $\underline{z}=(z_1,z_2,\cdots,z_\ell) \in \C^\ell$ and where the sum is over all integer partitions $r= r_1\geq r_2 \geq \cdots \geq r_\ell \geq 0$. We assume that:
\begin{itemize}
\item $|w_r| \leq C\, a^{-(r|r)/2+ \alpha |r|}$ for some absolute constant $C$ and $\alpha <3/2$;
\item $g(a^{m_1},a^{m_2},\dots, a^{m_\ell}) = 0$ for all nonnegative integers $m_1, m_2, \dots, m_\ell$.
\end{itemize}
If $\alpha<1/2$, then $g\equiv0$ and $w_r =0$ for all $r$.
If $\alpha \in [1/2,3/2[$ and $w_{0,0,\dots,0}=0$, then $g\equiv0$ (therefore $w_r =0$ for all $r$).
\end{corollary}
\begin{proof}
If $\ell=1$, this is proved by the above lemma, since we have $\omega > \alpha -1/2$ in both cases under consideration. If $\ell\geq 2$, we fix $(m_2, \dots, m_\ell) \in \N^{\ell-1}$ and set 
$$
f(z)=\sum_{r_1} z^{r_1} \left(\sum_{r_1 \geq r_2 \geq \cdots \geq r_\ell} w_{r_1,r_2,\dots,r_\ell} a^{r_2m_2}\cdots a^{r_\ell m_\ell} \right).  
$$
Then $f(z)$ satisfies the condition of the previous lemma since
\begin{eqnarray*}
\left|\sum_{r_2\geq \dots\geq r_\ell\geq0  } w_{r_1,r_2,\dots,r_\ell} a^{r_2m_2}\cdots a^{r_\ell m_\ell}\right| &&\ll_{m_2,\dots,m_\ell} a^{-(r_1|r_1)/2+\alpha r_1}.
\end{eqnarray*}
With the conditions of the corollary, we deduce that $f(z)=0$, therefore for any fixed $r_1\geq0$, we must have
$$
\sum_{r_1\geq r_2 \geq \cdots \geq r_\ell\geq0}  w_{r_1,r_2,\dots,r_\ell} a^{r_2m_2}\cdots a^{r_\ell m_\ell} = 0,
$$
for all $m_2,\dots, m_\ell$. We use the fact that when $r_1$ is fixed, 
$$
\ds \sum_{r_1\geq r_2 \geq \cdots \geq r_\ell\geq0}  w_{r_1,r_2,\dots,r_\ell} z_2^{r_2}\cdots z_r^{r_\ell}
$$
is a polynomial to conclude.
\end{proof}

\section{Class groups of number fields}
We will actually prove a more general result displayed in Theorem~\ref{general_clgp} (which clearly implies Theorem~\ref{class-group}). If $K$ is a number field, we denote by $\cl(K)$ its class group. Let ${\mathcal K}$ be a fixed set of number fields 
ordered by the absolute value of their discriminant $\disc(K)$. If $f$ is a real valued function defined over isomorphism classes of finite abelian 
groups, then, as before, 
we say that $f(\cl(K))$ has average value $c \in \R$ for $K$ varying over ${\mathcal K}$ if  
$$
\sum_{\substack{K\in { \mathcal K} \\ |\disc(K)|<X}} f(\cl(K))= (c+o(1)) \sum_{\substack{K\in { \mathcal K} \\ |\disc(K)|<X}} 1 , \quad  \mbox{ as } X \rightarrow \infty.
$$
As before, if $f$ is the characteristic function of some property, we say that $c$ is the probability of this property (or the density of the 
set of the class groups satisfying it) for $K$ varying in ${\mathcal K}$.
\begin{theorem}\label{general_clgp}
Let $u$ be a nonnegative integer and let $\ell$ be a positive integer. Assume 
that for every integer partition $\lambda~=~1^{m_1}2^{m_2}\cdots \ell^{m_\ell}$, as $K$ is varying over ${\mathcal K}$, the average of $|\cl(K)[p]|^{m_1} \cdots |\cl(K)[p^\ell]|^{m_\ell}$ is equal to $\sum_{\mu \subseteq \la} C_{\la,\mu}(p)p^{-u|\mu|}$. Then for any $\mu_1\geq \mu_2\geq \cdots \geq \mu_\ell$, as  $K$ is varying over ${\mathcal K}$, the probability that 
$\rk_{p^j}(\cl(K))=\mu_j$ for all $1\leq j \leq \ell$, is equal to
$$
\frac{(1/p^{u+\mu_\ell+1};1/p)_\infty}{p^{\mu_1^2+\cdots + \mu_\ell^2+u(\mu_1+\cdots + \mu_\ell)}  \prod_{j=1}^{\ell} (1/p;1/p)_{\mu_j-\mu_{j+1}}}.
$$
\end{theorem}
\medskip
~\\
We follow and generalize the proof of \cite{kluners_fouvry2}. First, we will need the following proposition.
\begin{proposition}\label{estimation_1}
Let $u\in \N$ and $\ell \in \N$. For all $\lambda=1^{m_1}2^{m_2}\cdots \ell^{m_\ell}$ we have 
\begin{equation}\label{polynomial_C}
\sum_{\mu \subseteq \la} C_{\la,\mu}(p)p^{-u|\mu|} = O_{p,\ell}(p^{(\lambda' | \lambda') /2}).
\end{equation}
\end{proposition}
\begin{proof}
Set $C_\lambda:= \sum_{\mu \subseteq \la} C_{\la,\mu}(p)p^{-u|\mu|}$. The equality  
$$
\qbi{n}{k}{p}= \frac{(p;p)_n}{(p;p)_k (p;p)_{n-k}} = p^{k(n-k)} \frac{(1/p;1/p)_n}{(1/p;1/p)_k (1/p;1/p)_{n-k}},
$$
and $(1/p;1/p)_{\infty} \leq(1/p;1/p)_k\leq 1$, together with the expression \eqref{coeff} of the coefficients $C_{\la,\mu}(p)$, imply that
$$
C_\lambda \leq C \sum_{\mu \subseteq \lambda} p^{\sum_i \mu'_i(\lambda'_i-\mu'_i)} \leq C\, p^{\sum_i \lambda_i'^2/4} \sum_{\mu \subseteq \lambda} 1,
$$
for some constant $C$ depending only on $p$ and $\ell$ (we used $\mu'_i(\lambda'_i-\mu'_i) \leq \lambda_i'^2/4$ for all $0\leq \mu_i' \leq \lambda'_i$, 
noting that $\mu \subseteq \lambda$ if and only if $\mu' \subseteq \lambda'$).
Now, since $\ell$ is fixed, the number of sub-partitions $\mu \subseteq \lambda$ is certainly bounded by the product 
$(\lambda_1+1)(\lambda_2+1)\cdots (\lambda_\ell+1)$. By the arithmetico-geometric mean inequality, we obtain 
$$
\sum_{\mu \subseteq \lambda} 1 \leq \left(1+\frac{|\lambda|}{\ell}\right)^\ell = O_\ell(|\lambda|^\ell).
$$
Finally, we have
$
C_\lambda = O_{p,\ell}(p^{(\lambda'|\lambda')/4} |\lambda|^\ell) = O_{p,\ell}(p^{(\lambda'|\lambda')/2}).
$
\end{proof}
~\\
{\bf Remark.} As it can be seen in the proof of the above proposition, we have the more precise upper bound 
$C_\lambda=O_{p,\ell}(p^{(\lambda'|\lambda')/4}|\lambda|^\ell)$. Nevertheless, the upper bound given in the proposition will be sufficient for our application. 
\medskip
~\\
\begin{proof}[Proof of Theorem~\ref{general_clgp}]
For $X\geq 1$ and $r:=r_1\geq r_2 \geq \cdots \geq r_\ell \geq 0$,  set 
\begin{eqnarray*}
N(X)  & := &|\{K \in {\mathcal K} \colon |\disc(K)|\leq X\}|, \mbox{ and } \\
N(X,r)& := &|\{K \in {\mathcal K} \colon |\disc(K)| \leq X \mbox{ and } \rk_{p^i}(\cl(K))=r_i\, \mbox{ for all } i=1,\dots,\ell \}|.
\end{eqnarray*}
For every $\lambda=1^{m_1}2^{m_2}\cdots \ell^{m_\ell}$, the assertion of Theorem \ref{general_clgp} implies that 
\begin{equation}\label{eq_1}
\sum_{r} \frac{N(X,r)}{N(X)}p^{m_1r_1+m_2(r_1+r_2)+\cdots + m_\ell(r_1+r_2+\cdots + r_\ell)} = \sum_{\mu\subseteq \lambda} C_{\lambda,\mu}(p)p^{-|\mu|u} +o_{\lambda}(1),
\end{equation}
where the sum on the left  is over all integer partitions $r=r_1\geq r_2 \geq \cdots \geq r_\ell \geq 0$. Note that 
$m_1r_1+m_2(r_1+r_2)+\cdots + m_\ell(r_1+r_2+\cdots + r_\ell) = (\lambda' | r)$. Hence, 
equation~\eqref{eq_1} can be written as
\begin{equation}\label{eq_2}
\sum_{r} \frac{N(X,r)}{N(X)}p^{(\lambda'|r)} = \sum_{\mu\subseteq \lambda} C_{\lambda,\mu}(p)p^{-|\mu|u} +o_{\lambda}(1), \quad (X\rightarrow \infty).
\end{equation}
For each integer partition $r$, the sequence $n\mapsto N(n,r)/N(n)$ is a real sequence in the compact set $[0,1]$. We deduce that there exists a
real number $d_r \in [0,1]$ and an infinite subset ${\mathcal M}$ of $\mathbb{N}$ such that 
$$
N(m,r)/N(m) \rightarrow d_r \quad (m\in {\mathcal M}, m \rightarrow \infty).
$$
Replacing $m_i$ by $2m_i+1$, we see from \eqref{eq_1} and Proposition~\ref{estimation_1} that 
\begin{equation}\label{inegalite_sha}
\frac{N(X,r)}{N(X)} \ll_\lambda p^{-(2m_1+1)r_1-(2m_2+1)(r_1+r_2)-\cdots -(2m_\ell+1)(r_1+\cdots+r_\ell)},  
\end{equation}
uniformly in $X$ and $r$, from which we deduce that 
$$
\sum_r \frac{N(m,r)}{N(m)} p^{(\lambda' | r)} = O_\lambda(1).
$$
Hence by Lebesgue's Dominated Convergence Theorem we have 
$$
\sum_r d_r p^{(\lambda'|r)} = \sum_{\mu\subseteq \lambda} C_{\lambda,\mu}(p)p^{-|\mu|u}.
$$
If we consider the infinite multi-dimensional system
\begin{equation}\label{systeme}\tag{$\mathcal S$}
 \sum_r x_r p^{(\lambda'|r)} = \sum_{\mu\subseteq \lambda} C_{\lambda,\mu}(p)p^{-|\mu|u} \mbox{ for all } \lambda=1^{m_1}\cdots \ell^{m_\ell},
\end{equation}
where the unknowns are $x_r\geq 0$ then $d_r$ is a solution of \eqref{systeme}. We already know that 
$$
x_r = \frac{(1/p^{\mu_\ell+u+1};1/p)_\infty}{p^{\mu_1^2+\cdots + \mu_\ell^2+u(\mu_1+\cdots + \mu_\ell)}  \prod_{j=1}^{\ell} (1/p;1/p)_{\mu_j-\mu_{j+1}}}
$$
is a solution (see \cite{delaunay_jouhet}, Theorem 14 or the equality just before this theorem). Therefore we need to prove that there exists at most one solution to the system.~\medskip
~\\
Let $(x_r)_r$ be a solution of \eqref{systeme}. Since $\lambda \mapsto \lambda'$ is a bijection, the system is equivalent to 
$$
\sum_r x_r p^{(\lambda|r)} = C_{\lambda'} \mbox{ for all } \lambda,
$$
where $C_{\lambda'}(p) = \sum_{\mu\subseteq \lambda'} C_{\lambda',\mu}(p)p^{-|\mu|u} = O(p^{(\lambda|\lambda)/2})$ by 
Proposition~\ref{estimation_1}. From $x_r\geq 0$, we deduce $x_r= O(p^{-(\lambda|r)+(\lambda|\lambda)/2})$, so when $\lambda=r$, we get
$$
0\leq x_r \leq c_0 p^{-(r|r)/2},
$$
for some absolute constant $c_0$. Now, if $(x'_r)_r$ is another solution of \eqref{systeme}, then setting $w_r=x_r-x'_r$, we have a function 
\begin{equation}\label{f}
f(\underline{z}) = f(z_1,z_2,\dots, z_\ell) = \sum_r w_r z_1^{r_1} z_2^{r_2}\cdots z_\ell^{r_\ell}
\end{equation}
satisfying $f(\underline{z})=0$ if $z_1=p^{m_1}, z_2=p^{m_2}, \dots, z_\ell=p^{m_\ell}$ for all $m_1, m_2,\dots, m_\ell \in \N$, and 
$$
|w_r| \leq 2c_0 p^{-(r|r)/2}.
$$
Thus we can apply Corollary~\ref{coro_2} to conclude that $x_r=x'_r$. So, as $X\rightarrow \infty$, the sequence $N(X,r)/N(X)$ has only one limit point, which is $d_r=x_r$.
\end{proof}

\begin{corollary}\label{coro1}
Let $u$ be a nonnegative integer and let $\ell$ be a positive integer. Assume 
that for every integer partition $\lambda~=~1^{m_1}2^{m_2}\cdots \ell^{m_\ell}$, as $K$ is varying over ${\mathcal K}$, the average of $|\cl(K)[p]|^{m_1} \cdots |\cl(K)[p^\ell]|^{m_\ell}$ is equal to $\sum_{\mu \subseteq \la} C_{\la,\mu}(p)p^{-u|\mu|}$. Then for $k\in \N$, as  $K$ is varying over 
${\mathcal K}$, the probability that $\rk_{p^\ell}(\cl(K))=k$ is equal to
$$
\frac{(1/p^{u+k+1};1/p)_\infty}{(1/p;1/p)_k p^{\ell k (u+k)}}\, Q_{1/p,\ell,1}\left(\frac{1}{p^{2k+u-1}}\right),
$$
where $\displaystyle Q_{q,\ell,1}(x):=\sum_{n\geq 0} \frac{(-1)^n x^{\ell n} q^{n(n+1)(2\ell+1)/2-n}(1-xq^{2n+1})}{(q;q)_n(xq^{n+1};q)_\infty}$.
\end{corollary}
\noindent
The series $Q_{q,\ell,k}(x) $ was defined by Andrews (see \cite{andrews}). The formula of the above corollary is the $u$-probability that the 
$p^\ell$-rank of a finite abelian $p$-group is equal to $k$, as obtained in \cite{cohen3} (note that we use the definition of $u$-averages and $u$-probabilities 
of this article).
\begin{proof}
We define as before $N(X)$ and $N(X,r)$. Moreover, set
$$
N(X,\ell,k)=|\{K \in {\mathcal K} \colon |\disc(K)| \leq X \mbox{ and } \rk_{p^\ell}(\cl(K))=k \}|.
$$
We have 
\begin{equation}\label{inter}
\frac{N(X,\ell,k)}{N(X)} = \sum_{\mu_1\geq \mu_2 \geq \cdots \geq \mu_{\ell-1} \geq k} \frac{N(X,\mu)}{N(X)},
\end{equation}
where the sum is over integer partitions $\mu=\mu_1\geq \mu_2 \geq \cdots \geq \mu_{\ell-1} \geq \mu_\ell=k$. By the assumptions and 
Theorem~\ref{general_clgp}, we have 
$$
\lim_{X\rightarrow \infty} \frac{N(X,\mu)}{N(X)} = \frac{(1/p^{u+\mu_\ell+1};1/p)_\infty}{p^{\mu_1^2+\cdots + \mu_\ell^2+u(\mu_1+\cdots + \mu_\ell)}  \prod_{j=1}^{\ell} (1/p;1/p)_{\mu_j-\mu_{j+1}}}.
$$
In equation~\eqref{inter} we take the limit as $X \rightarrow \infty$ and use Lebesgue's Dominated Convergence Theorem (with 
equation~\eqref{inegalite_sha}) to obtain that the probability that $\rk_{p^\ell}(\cl(K))=k$ is equal to
\begin{multline*}
\sum_{\mu_1\geq \cdots \geq \mu_{\ell-1} \geq \mu_\ell=k} \frac{(1/p^{u+\mu_\ell+1};1/p)_\infty}{p^{\mu_1^2+\cdots + \mu_\ell^2+u(\mu_1+\cdots + \mu_\ell)}  \prod_{j=1}^{\ell} (1/p;1/p)_{\mu_j-\mu_{j+1}}} \\ \\
=\frac{(1/p^{u+k+1};1/p)_\infty}{p^{\ell k(u+k)}(1/p;1/p)_k} \sum_{\mu_1\geq \cdots  \geq \mu_{\ell-1}\geq 0} \frac{(1/p)^{\mu_1^2+\cdots + \mu_{\ell-1}^2+(u+2k)(\mu_1+\cdots + \mu_{\ell-1})}}{  \prod_{j=1}^{\ell-1} (1/p;1/p)_{\mu_j-\mu_{j+1}}},
\end{multline*}
the equality being derived by shifting all indices by $k$.
Now, from \cite[Proposition 13]{delaunay3} (or \cite[Eq. (2.5)]{andrews}), the last sum is exactly
$$
Q_{1/p,\ell,1}\left(\frac{1}{p^{2k+u-1}}\right),
$$
as expected (note that $Q_{q,1,1}(x)=1$).
\end{proof}

\section{Tate-Shafarevich groups of elliptic curves}
In this section, we prove Theorem~\ref{general_sha} and we follow the previous one. We just need an upper bound for the coefficients 
$\sum_{\mu \subseteq \la} C_{\lambda,\mu}(p^2)p^{-|\mu|(2u-1)}$, which is given in the following result.
\begin{proposition}\label{estimation_2}
Let $u\in \N$ and $\ell \in \N^*$. For all $\lambda=1^{m_1}2^{m_2}\cdots \ell^{m_\ell}$, we have 
$$
\sum_{\mu \subseteq \la}  C_{\lambda,\mu}(p^2)p^{-|\mu|(2u-1)}  = O_{p,\ell}(p^{(\lambda' | \lambda')}).
$$
\end{proposition}
\begin{proof}
We have 
\begin{eqnarray*}
\sum_{\mu \subseteq \la}  C_{\lambda,\mu}(p^2)p^{-|\mu|(2u-1)} &=& \sum_{\mu \subseteq \la} p^{2(\sum_i \mu_{i+1}'(\lambda_i'-\mu_i'))} \prod_i \qbi{\lambda_i'-\mu_{i+1}'}{\lambda_i'-\mu_i'}{p^2} p^{-|\mu|(2u-1)} \\
&\leq& p^{|\lambda|} \sum_{\mu \subseteq \la}  p^{2(\sum_i \mu_{i+1}'(\lambda_i'-\mu_i'))} \prod_i \qbi{\lambda_i'-\mu_{i+1}'}{\lambda_i'-\mu_i'}{p^2}.
\end{eqnarray*}
Using the same method as in the proof of Proposition~\ref{estimation_1}, we obtain 
\begin{equation}\label{proba_best}
\sum_{\mu \subseteq \la}  C_{\lambda,\mu}(p^2)p^{-|\mu|(2u-1)} = O_{p,\ell}(p^{(\lambda'|\lambda')/2 + |\lambda|} |\lambda|^\ell),
\end{equation}
which implies the upper bound given in the proposition.
\end{proof}
\medskip
\noindent
\begin{proof}[Proof of Theorem~\ref{general_sha}]
For $r:=r_1\geq r_2 \geq \cdots \geq r_\ell \geq 0$, we denote by $2r$ the integer partition $2r:=2r_1\geq 2r_2 \geq \cdots \geq 2r_\ell$. For $X\geq 1$, set 
\begin{eqnarray*}
N(X)  & := &|\{E \in {\mathcal F} \colon N_E\leq X\}|, \mbox{ and } \\
N(X,r)& := &|\{E \in {\mathcal F} \colon N_E \leq X \mbox{ and } \rk_{p^i}(\Sha(E))=r_i\,\mbox{ for all } i=1,\dots,\ell \}|.
\end{eqnarray*}
Note that since $\Sha(E)$ is a group of type S then $\rk_{p^j}(\Sha(E))$ must be even. Hence if one of the $r_j$'s is odd, then 
$N(X,r)=0$ for all $X$. 
So, for any $\lambda=1^{m_1}2^{m_2}\cdots \ell^{m_\ell}$, the assertion of Theorem \ref{general_sha} implies that 
\begin{equation}\label{eq_3}
\sum_{r} \frac{N(X,2r)}{N(X)}p^{(\lambda' | 2r)} = \sum_{\mu\subseteq \lambda} C_{\lambda,\mu}(p^2)p^{-|\mu|(2u-1)} +o_{\lambda}(1),
\end{equation}
where the sum is over all integer partition $r=r_1\geq \cdots \geq r_\ell$. As before, we will prove that the sequence $N(X,2r)/N(X)$ has only one limit point 
as $X \rightarrow \infty$. We are led to consider the system
\begin{equation}\label{systeme_2}\tag{$\mathcal T$}
 \sum_r x_{2r} p^{(\lambda'|2r)} =  \sum_{\mu\subseteq \lambda} C_{\lambda,\mu}(p^2)p^{-|\mu|(2u-1)} \mbox{ for all } \lambda=1^{m_1}\cdots \ell^{m_\ell},
\end{equation}
where the unknowns are $x_{2r}\geq 0$. If $e_{2r}$ is a limit point of $N(X,2r)/N(X)$, then $e_{2r}$ is a solution of \eqref{systeme_2}. 
We already know that 
$$
x_{2r}=\frac{(1/p^{2u+2\mu_\ell+1};1/p^2)_\infty}{p^{2(\mu_1^2+\cdots + \mu_\ell^2)+(2u-1)(\mu_1+\cdots + \mu_\ell)} \prod_{j=1}^{\ell} (1/p^2;1/p^2)_{\mu_j-\mu_{j+1}}}
$$
is a solution (see \cite{delaunay_jouhet}, remark after Theorem 14). If $x_{2r}$ is a solution 
of \eqref{systeme_2}, then we have 
$$
\sum_r x_{2r} p^{2(\lambda|r)} = \sum_{\mu\subseteq \lambda} C_{\lambda',\mu}(p^2)p^{-|\mu|(2u-1)}  \mbox{ for all } \lambda,
$$
where $\sum_{\mu\subseteq \lambda} C_{\lambda',\mu}(p^2)p^{-|\mu|(2u-1)} = O(p^{(\lambda|\lambda)})$, from which we deduce
$$
x_{2r} \ll p^{(r|r)-2(r|r)} \ll p^{-(r|r)} = (p^{2})^{-(r|r)/2}.
$$
Therefore Corollary~\ref{coro_2} with $a=p^2$ gives the unicity. 
\end{proof}
\begin{corollary}
Let $u$ be a nonnegative integer, and let $\ell$ be a positive integer. Assume that for 
any $\lambda=1^{m_1}2^{m_2}\cdots \ell^{m_\ell}$, as $E/\Q$, ordered by conductors, is varying over ${\mathcal F}_u$, 
the average of $|\Sha(E)[p]|^{m_1} |\Sha(E)[p^2]|^{m_2} \cdots |\Sha(E)[p^\ell]|^{m_\ell}$ is equal to 
$\sum_{\mu \subseteq \la} C_{\lambda,\mu}(p^2)p^{-|\mu|(2u-1)}$. Then for $k\in \N$, as $E/\Q$ is varying over ${\mathcal F}_u$, the probability that 
$\rk_{p^\ell}(\Sha(E))=2k$ is equal to
$$
\frac{(1/p^{2u+2k+1};1/p^2)_\infty}{(1/p^2;1p^2)_k p^{\ell k(2u+2k-1)}}\, Q_{1/p^2,\ell,1}(1/p^{4k+2u-3}).
$$
\end{corollary}
\begin{proof}
We proceed as for the proof of Corollary \ref{coro1}.
\end{proof}
\noindent
The formula in the above corollary is the $u$-probability that the 
$p^\ell$-rank of a finite abelian $p$-group of type S is equal to $2k$, as obtained in \cite{delaunay3}.

\section{Selmer groups of elliptic curves}\label{Selmer}
In the proof of Theorem \ref{general_sha}, it is essential to use the fact that $\Sha(E)$ is a group of type~S and that $\rk_{p^j}(\Sha(E))$ is even, 
since on the left hand side of \eqref{eq_3} the sum involves partitions with only even parts $\mu_j$. Nevertheless, one can ask what should be 
the $p^j$-rank probability laws for other families of groups, if we assume that their moments are given as in Conjecture \ref{moment_sha}. 
This question can be naturally asked in particular for Selmer groups of elliptic curves (or more precisely for the $p$-primary parts of the Selmer groups). 
If $E$ is an elliptic curve defined over $\Q$, then we denote by $S(E)$ the $p$-primary part of its Selmer group. It is the inductive limit of 
the $p^n$-Selmer group $S(E)_{p^n}$ of $E$:
$$
S(E) = \lim_{\longrightarrow} S(E)_{p^n}.
$$
We have the exact sequence
$$
0 \rightarrow E(\Q)\otimes \Q_p/\Z_p \rightarrow S(E) \rightarrow \Sha(E)[p^\infty] \rightarrow 0,
$$
which can be seen as the limit of 
$$
0 \rightarrow E(\Q)/p^nE(\Q) \rightarrow S(E)_{p^n} \rightarrow \Sha(E)[p^n] \rightarrow 0.
$$
We assume for simplicity that $E(\Q)_{\rm{tors}}$ is trivial: it is not a restriction since we are considering averaging over elliptic curves and on 
average elliptic curves have trivial rational torsion. The Selmer group $S(E)$ can be an infinite group, nevertheless its subgroup of $p^n$-torsion 
points is finite and we have 
$$
S(E)[p^n] = S(E)_{p^n}.
$$
We define the $p^j$-rank of $S(E)$ by $\rk_{p^j}(S(E))=\rk_{p^j}(S(E)[p^j])$. Note that we have $\rk_{p^j}(S(E))=\rk_{p^j}(S(E)[p^k])$ for all $k\geq j$.
~\\
Since $\Sha[p^\infty]$ is finite by assumption, we have $S(E)\simeq (\Q_p/\Z_p)^{r(E)}$, where $r(E)$ is the rank of the Mordell-Weil group $E(\Q)$ and
$$
\rk_{p^j}(S(E)) = r(E) \quad \mbox{ for } p^j \mbox{ large enough.}
$$
Furthermore, we have 
$$
\rk_{p^j}(S(E)) \equiv r(E) \pmod{2},
$$
so the parities of $\rk_{p^j}(S(E))$ are determined by the parity of $r(E)$.\medskip
~\\
If $\ell$ is a positive integer and $\lambda~=~1^{m_1}2^{m_2}\cdots \ell^{m_\ell}$ is an integer partition, then 
$$
|S(E)[p]|^{m_1} |S(E)[p^2]|^{m_2} \cdots |S(E)[p^\ell]|^{m_\ell}
$$
is meaningful and we can consider the average value of this function as $E$ is varying over a family of elliptic curves.
The works of \cite{poonen-rains} suggests that the $p$-Selmer groups should behave in a ``global'' way independently of the rank 
of $E$ (except for the parity of the $p$-ranks). From \cite{delaunay_jouhet}, we can extract the following conjecture.
\begin{conjecture}\label{conj_selmer}
Let $\ell$ be a positive integer, and let $\lambda=1^{m_1}2^{m_2}\cdots \ell^{m_\ell}$ be an integer partition. As $E/\Q$, 
ordered by conductors, is varying over all elliptic curves, the average of 
$|S(E)[p]|^{m_1} |S(E)[p^2]|^{m_2} \cdots |S(E)[p^\ell]|^{m_\ell}$ is equal to
$$
\sum_{\mu \subseteq \lambda} C_{\lambda,\mu}(p^2)p^{|\mu|}.
$$
\end{conjecture}
\medskip
~\\
If $\ell=1$, this conjecture is originally due to Poonen and Rains in~\cite{poonen-rains}, where they use a completely different model for Selmer group. 
\begin{proposition}\label{selmer}
Let $\ell$ be a positive integer and let $\delta \in \{0,1\}$. Assume that for any partition $\lambda=1^{m_1}2^{m_2}\cdots \ell^{m_\ell}$, the average of 
$|S(E)[p]|^{m_1} |S(E)[p^2]|^{m_2} \cdots |S(E)[p^\ell]|^{m_\ell}$ is equal to 
$\sum_{\mu \subseteq \la} C_{\lambda,\mu}(p^2)p^{|\mu|}$ as $E/\Q$, ordered by conductor, is varying over a family ${\mathcal F}$ of elliptic curves, and assuming that the even (resp. odd) rank elliptic curves in ${\mathcal F}$ contribute in a ratio $\alpha$ (resp. $1-\alpha$).
Then, for all $\mu_1\geq \mu_2 \cdots \geq \mu_\ell$, as $E/\Q$ is varying over ${\mathcal F}$, the probability that $\rk_{p^j}(S(E))=2\mu_j+\delta$ for all $1\leq j \leq \ell$ 
 is equal to 
$$
 \left(\delta(1-\alpha)+\alpha(1-\delta)\right) \frac{(1/p^{2\delta+2\mu_\ell+1};1/p^2)_\infty}{p^{2(\mu_1^2+\cdots + \mu_\ell^2)+(2\delta-1)(\mu_1+\cdots + \mu_\ell)} \prod_{j=1}^{\ell} (1/p^2;1/p^2)_{\mu_j-\mu_{j+1}}}.
$$
\end{proposition}
\begin{proof}
For $X\geq 1$ and $r=r_1\geq r_2 \geq \cdots \geq r_\ell \geq 0$,  set as before
\begin{eqnarray*}
N(X)  & := &|\{E \in {\mathcal F} \colon N_E\leq X\}|, \mbox{ and } \\
N(X,r)& := &|\{E \in {\mathcal F} \colon N_E \leq X \mbox{ and } \rk_{p^i}(S(E))=r_i\,\mbox{ for all } i=1,\dots,\ell \}|.
\end{eqnarray*}
Let $\lambda=1^{m_1}\cdots \ell^{m_\ell}$ be an integer partition. Since the $\rk_{p^j}(S(E))$'s have all the same parity for $j\in \N$,  and by the 
assumptions of the theorem, we have 
$$
\sum_r \frac{N(X,2r)}{N(X)}p^{(\lambda'|2r)} = \alpha \sum_{\mu\subseteq \lambda} C_{\lambda,\mu} p^{|\mu|} + o_{\lambda}(1), 
$$
and 
$$
\sum_r \frac{N(X,2r+1)}{N(X)} p^{(\lambda'|2r+1)} =(1-\alpha)\sum_{\mu\subseteq \lambda} C_{\lambda,\mu} p^{|\mu|} + o_{\lambda}(1).
$$
For $\delta \in \{0,1\}$, set
$$
e_{2r+\delta} = \frac{(1/p^{2\delta+2\mu_\ell+1};1/p^2)_\infty}{p^{2(\mu_1^2+\cdots + \mu_\ell^2)+(2\delta-1)(\mu_1+\cdots + \mu_\ell)} \prod_{j=1}^{\ell} (1/p^2;1/p^2)_{\mu_j-\mu_{j+1}}}.
$$
Thus for $\delta=0$, we recover $e_{2r}$ which was defined in the previous section, where we already saw that
$$
\sum_{r} \alpha e_{2r} p^{(\lambda'|2r)} = \alpha \sum_{\mu\subseteq \lambda} C_{\lambda,\mu}(p^2)p^{|\mu|},
$$
 that $\alpha e_{2r}$ is the only solution of the above system, and that $N(X,2r)/N(X) \rightarrow e_{2r}$ as $X\rightarrow \infty$.
~\\
Now, by the same arguments as before, we obtain that there exists an unique solution to the system
\begin{equation}\label{systeme_4}\tag{$\mathcal T$'}
\sum_r \left(x_{2r+1} p^{(\lambda'|2r+1)}\right) = (1-\alpha) \sum_{\mu\subseteq \lambda} C_{\lambda,\mu}(p^2)p^{|\mu|}, \mbox{ for all } \lambda=1^{m_1}\cdots \ell^{m_\ell},
\end{equation}
where the unknown are $x_{2r+1}$. Furthermore $x_{2r+1}=(1-\alpha)e_{2r+1}$ is a solution of~(\ref{systeme_4}). Indeed,  
by \cite[Remark after Theorem 14]{delaunay_jouhet}
\begin{eqnarray*}
\sum_{r} e_{2r+1} p^{(\lambda'|2r+1)} &=&  p^{|\lambda'|}\sum_{r} e_{2r+1} p^{(\lambda'|2r)} \\
&=&p^{|\lambda|}  \sum_{\mu\subseteq \lambda} C_{\lambda,\mu}(p^2)p^{-|\mu|},
\end{eqnarray*}
and moreover by \cite[Theorem 1]{delaunay_jouhet}, we have 
$$
p^{|\lambda|} \sum_{\mu\subseteq \lambda} C_{\lambda,\mu}(p^2)p^{-|\mu|} = \sum_{\mu\subseteq \lambda} C_{\lambda,\mu}(p^2)p^{|\mu|}.
$$
Finally $N(X,2r+1)/N(X) \rightarrow e_{2r+1}$ as $X\rightarrow \infty$.
\end{proof}
\medskip
~\\
Now, adapting the proof of Corollary \ref{coro1}, we have the following result.
\begin{corollary}\label{th_Selmer}
Let $\ell$ be a positive integer and set $\delta \in \{0,1\}$. Assume that for every $\lambda=1^{m_1}2^{m_2}\cdots \ell^{m_\ell}$ 
the average of $|S(E)[p]|^{m_1} |S(E)[p^2]|^{m_2} \cdots |S(E)[p^\ell]|^{m_\ell}$ is equal to 
$\sum_{\mu \subseteq \la} C_{\lambda,\mu}(p^2)p^{|\mu|}$, as $E/\Q$, ordered by conductor, is varying over a family ${\mathcal F}$ of elliptic curves, and assuming that the even (resp. odd) rank elliptic curves in ${\mathcal F}$ contribute in a ratio $\alpha$ (resp. $1-\alpha$).
Then, for $k\in \N$, the probability that $\rk_{p^\ell}(S(E))= 2k+\delta$ is equal to
$$
\left(\delta(1-\alpha)+\alpha(1-\delta)\right) \frac{(1/p^{2k+2\delta+1};1/p^2)_\infty}{(1/p^2;1/p^2)_k p^{\ell k(2k+2\delta-1)}} Q_{1/p^2,\ell,1}(1/p^{4k+2\delta-3}).
$$
\end{corollary}
\noindent
\medskip
~\\
The value of $\alpha$ can be of course $\neq 1/2$. Furthermore, even in the case of a family of quadratic twists of an elliptic curve $E$ 
defined over a number field $K$, it is possible to have $\alpha\neq 1/2$ (see \cite{klagsbrun_all_1}, in that case we have $K\neq \Q$).~\\
If we consider the family of all elliptic curves, then a general conjecture states that $\alpha=1/2$, which leads to the following.
\begin{conjecture} Let $\ell$ be a positive integer, set $k\in \N$ and $\delta \in \{0,1\}$. Then, as $E/\Q$, ordered by conductor, is varying over all 
elliptic curves, the probability that $\rk_{p^\ell}(S(E))= 2k+\delta$ is equal to
$$ 
f(p,\ell,2k+\delta) : = \frac{1}{2}\frac{(1/p^{2k+2\delta+1};1/p^2)_\infty}{(1/p^2;1/p^2)_k p^{\ell k(2k+2\delta-1)}} Q_{1/p^2,\ell,1}(1/p^{4k+2\delta-3}).
$$
\end{conjecture}
\noindent
For $\ell=1$, we recover the conjectural distribution $X_{\sel_p}$ of \cite{poonen-rains} and the proved distribution of $\sel_2$ in 
 \cite{heath-brown, heath-brown2, kane, swinnerton-dyer, klagsbrun_all_2} for some families of quadratic twists of an elliptic curve. \medskip
~\\
The conjectural distribution on the $p^j$-rank of the Selmer groups given above is of course compatible with the rank conjecture. Indeed, note that we have $Q_{q, \infty,1}(x)= 1/(xq^2;q)_\infty$ from which we easily deduce that  
$$
\lim_{\ell \rightarrow \infty}  f(p,\ell,2k+\delta) = \left\{ \begin{array}{cc} 0 & \mbox{ if } k\geq 1, \\
											   1/2 & \mbox{ if } k=0. \end{array} \right. 
$$
Since for $\ell$ large enough, $\rk_{p^\ell}(S(E))=r(E)$, we recover the fact that on average half elliptic curves should have rank 0 and half 
elliptic curves should have rank 1.
On the other hand, if we assume Conjecture \ref{conj_selmer} for $\ell=1$ for infinitely many primes $p$ with $\alpha=1/2$, then we 
also recover the previous distribution for the rank of $E(\Q)$, since 
$$
\lim_{p \rightarrow \infty}  f(p,1,2k+\delta) = \left\{ \begin{array}{cc} 0 & \mbox{ if } k\geq 1, \\																							   1/2 & \mbox{ if } k=0. \end{array} \right. 
$$
\medskip
~\\
We give some numerical approximations for the function $f(p,\ell,2k+\delta)$ for $p=2, 3, 5$ and for small values of $\ell$ and of $2k+\delta$ 
in the following tables.
~\\
\begin{small}
\begin{center}
\begin{tabular}{cc}
\begin{tabular}{|c|c|c|c|}
\hline
$2k+\delta \backslash \ell  $ & 1 & 2 & 3 \\
\hline
0            & 0.2097 & 0.3541 & 0.4271 \\
\hline
1           & 0.4194 & 0.4899 & 0.4987 \\
\hline
2           & 0.2796 & 0.1456 & 0.0729 \\
\hline
3           & 0.0798 & 0.1009 & 0.0012\\
\hline
\end{tabular} &
\begin{tabular}{|c|c|c|c|}
\hline
$2k+\delta \backslash \ell  $ & 1 & 2 & 3 \\
\hline
0            & 0.3195 & 0.4398 & 0.4799 \\
\hline
1           & 0.4792 & 0.4992 & 0.4999 \\
\hline
2           & 0.1797 & 0.0601 & 0.0201 \\
\hline
3           & 0.0207 & 0.0007 & $2\cdot 10^{-5}$ \\
\hline
\end{tabular} \\
$p=2$ & $p=3$ 
\end{tabular}
\bigskip \bigskip
~\\
\begin{tabular}{cc}
\begin{tabular}{|c|c|c|c|}
\hline
$2k+\delta \backslash \ell  $ & 1 & 2 & 3 \\
\hline
0            & 0.3966 & 0.4793 & 0.4959 \\
\hline
1           & 0.4958 & 0.4999 & 0.4999 \\
\hline
2           & 0.1033 & 0.0207 & 0.0041 \\
\hline
3           & 0.0042 & $3\cdot 10^{-5}$ & $2\cdot 10^{-7}$\\
\hline
\end{tabular}
\\
$p=5$
\end{tabular}
\end{center}
\end{small}
\section{Remark on the uniqueness of the solution}
In our study related to Tate-Shafarevich group, we were led to consider and to discuss the unicity of the solution of 
the following infinite multi-dimensional system 
\begin{equation}\label{systeme_3}\tag{$\mathcal U$}
 \sum_r x_{r} p^{(\lambda|r)} =  \sum_{\mu\subseteq \lambda'} C_{\lambda',\mu}(p^2)p^{-|\mu|(2u-1)} \mbox{ for all } \lambda=1^{m_1}\cdots \ell^{m_\ell},
\end{equation}
where the unknowns are $x_{r}\geq 0$. We only considered solution $(x_r)_r$ such that $x_r=0$ if in $r=r_1\geq r_2 \geq \cdots \geq r_\ell$ at least 
one of the $r_j$'s has not the same parity as $r_1$. In that case, the term $p^{(\lambda|r)}$  involved in the sum is of the form 
$p^{(\lambda|2r+\delta)}$, and the factor 2 allowed to have an asymptotic $0\leq x_{2r+\delta} \ll p^{(-r|r)/2}$ which implied the unicity of the solution. 
One can ask about the unicity of the solution without the assumption that the partitions involved in the system have parts with the same parity.\medskip
~\\
Set $\mu \in \R$, and for a partition $r$ we define
$$
y_r(\mu)=\left\{ \begin{array}{cc} \mu e_{r} & \mbox{ if $r$ is even,}  \\ (1-\mu)e_{r} & \mbox{ if $r$ is odd,} \\ 0 &\mbox{otherwise.}  \end{array}\right.
$$
Then $y_r(\mu)$ is a solution of equation \eqref{systeme_3}. If $0\leq x_r$ is a solution of \eqref{systeme_3} then, using ~\eqref{proba_best}, 
we see that for any fixed $\alpha>1$, we have
$$
0\leq x_r \ll p^{-(r|r)/2+\alpha |r|} 
$$
for all $r$. Now, if we let $w_r=x_r- y_r(x_{0,0,\dots,0})$ then  it is easy to see that we also have 
$$
|w_r| \ll p^{-(r|r)/2+\alpha |r|}
$$
($w_r$ is not necessarily nonnegative), and the function $g(\underline{z})= \sum_r w_r z_1^{r_1}\cdots z_\ell^{r_\ell}$ satisfies the hypothesis 
of Corollary~\ref{coro_2} with $\alpha \in ]1,3/2[$ and with $w_{0,0,\dots,0}=0$. Hence $w_r=0$ for all $r$ and $x_r= y_r(x_{0,0,\dots,0})$. We  deduce the following proposition.
\begin{proposition} If $0\leq x_r$ is a solution of equation \eqref{systeme_3} then $x_r=y_r(\mu)$ for some $\mu$. In particular, 
$x_r=0$ if $r$ is not an even nor an odd partition.
\end{proposition}
~\\
It would be interesting to study the (unicity of the) solutions of equation \eqref{systeme_3} if we do not assume that $x_r \geq 0$. 
%

%
%

\bibliographystyle{alpha}
\bibliography{note}

\end{document}